\documentclass[12pt]{article}
\usepackage{amsthm,amsmath,amssymb,cite}

\newtheorem{thm}{Theorem}[section]
\newtheorem{lem}[thm]{Lemma}

\newtheorem{prb}[thm]{Problem}

\makeatletter \@addtoreset{equation}{section} \makeatother

\begin{document}

\begin{center}
{\Large\bf On the Existence of General Factors\\[8pt]
in Regular Graphs}
\end{center}

\begin{center}
Hongliang Lu$^{1}$, David G. L. Wang$^{2}$ and Qinglin Yu$^{3}$\\[6pt]

$^{1}$Department of Mathematics\\
$^{1}$Xi'an Jiaotong University, Xi'an 710049, P. R. China\\
{\tt $^{1}$luhongliang@mail.xjtu.edu.cn}

$^{2}$Beijing International Center for Mathematical Research\\
$^{2}$Peking University, Beijing 100871, P. R. China\\
{\tt $^{2}$wgl@math.pku.edu.cn}

$^{3}$Department of Mathematics and Statistics\\
$^{3}$Thompson Rivers University, Kamloops, BC, Canada\\
{\tt $^{3}$yu@tru.ca}
\end{center}

\begin{abstract}
Let $G$ be a graph, and $H\colon V(G)\to 2^\mathbb{N}$ a set function associated with~$G$. 
A spanning subgraph~$F$ of~$G$ is called an $H$-factor if the degree of any vertex~$v$
in~$F$ belongs to the set~$H(v)$. 
This paper contains two results on the existence of $H$-factors in regular graphs.
First, we construct an $r$-regular graph without some given $H^*$-factor.
In particular, this gives a negative answer to a problem recently posed by
Akbari and Kano.
Second, by using Lov\'asz's characterization theorem on the existence of $(g,f)$-factors,
we find a sharp condition for the existence of general $H$-factors in $\{r,\,r+1\}$-graphs,
in terms of the maximum and minimum of~$H$.
The result reduces to Thomassen's theorem
for the case that $H(v)$ consists of the same two consecutive integers for all vertices~$v$, 
and to Tutte's theorem if the graph is regular in addition.
\end{abstract}


\noindent\textbf{2010 AMS Classification:} 05C75

\section{Introduction}

Let~$G=(V(G),\,E(G))$ be a simple graph, where~$V(G)$ and~$E(G)$
denote the set of vertices and edges of~$G$
respectively. For any vertex~$v$, denote the degree of~$v$ by~$d_G(v)$.
Let~$2^\mathbb{N}$ denote the collection of sets of nonnegative integers.
We call
\[
H\colon V(G)\to 2^\mathbb{N}
\] 
a {\em set function associated with~$G$} if $H(v)\subseteq\{0,\,1,\,\ldots,\,d_G(v)\}$.
A spanning subgraph~$F$ of~$G$ is called an {\em $H$-factor} 
if $d_F(v)\in H(v)$ for all~$v$. 
It is often that~$H(v)$ coincides with some set~$H'$ for all~$v$.
In this case, we call~$H'$ a {\em set associated with~$G$},
and call~$F$ an {\em $H'$-factor} without confusion.
Let 
\[
g,\, f\colon V(G)\to\mathbb{Z}
\] 
be two functions such that $g(v)\le f(v)$ for all~$v$.
An $H$-factor is called a {\em $(g,f)$-factor} 
if~$H(v)$ is the interval $[g(v),f(v)]$ for all~$v$. 
A $(g,f)$-factor is called an {\em $[a,b]$-factor} 
if $g(v)=a$ and $f(v)=b$ for all~$v$. 
An $[a,b]$-factor~$F$ is called an {\em $(a,b)$-parity-factor}
if 
\[
d_F(v)\equiv a\equiv b\pmod 2\qquad\hbox{for every vertex }v.
\] 
In particular,
$F$~is called a {\em $k$-factor} if $a=b=k$.

A graph is said to be {\em $r$-regular} if every vertex has degree~$r$.
This paper is concerned with the existence of $H$-factors in regular graphs.
The study on the existence of factors in regular graphs was started, to the best of our knowledge, 
from Petersen~\cite{Pet1891}.

\begin{thm}[Petersen]\label{thm-Petersen}
Let $r$ and $k$ be even integers such that $1\le k\le r$. Then any $r$-regular
graph has a $k$-factor.
\end{thm}

In contrast with even-factors in Theorem~\ref{thm-Petersen}, 
Gallai~\cite{Gal50} obtained the next result for odd-factors.
For any graph~$G$, 
we call the number $|V(G)|$ of vertices the {\em order} of~$G$,
denoted alternatively by $|G|$ as usual.

\begin{thm}[Gallai]\label{thm-Gallai}
Let $r$, $k$ and $m$ be integers such that $r$ is even, $k$ is odd and
\[
{r\over m}\le k\le r\biggl(1-{1\over m}\biggr).
\]
Then any $m$-edge-connected $r$-regular graph of even order has a $k$-factor.
\end{thm}

It is clear that having an odd-factor implies that the order of the graph must be even.
So the ``even order'' condition in Theorem~\ref{thm-Gallai} is not a real restriction.
Removing the parity conditions for both $r$ and $k$, 
Tutte~\cite{Tutte78} gave the following theorem.

\begin{thm}[Tutte]\label{thm-Tutte}
Let $1\le k\le r-1$. Then any
$r$-regular graph has a $\{k,\,k+1\}$-factor.
\end{thm}

A graph~$G$ is said to be an {\em $\{r,\,r+1\}$-graph} if every vertex of~$G$ has degree~$r$ or~$r+1$.
Thomassen~\cite{Tho81} generalized Theorem~\ref{thm-Tutte}
by considering $\{r,\,r+1\}$-graphs.

\begin{thm}[Thomassen]\label{thm-Thomassen}
Let $1\le k\le r-1$.
Then any $\{r,\,r+1\}$-graph has a $\{k,\,k+1\}$-factor.
\end{thm}

For more results along this line, 
the reader is referred to Akiyama and Kano's book~\cite{AK11B}. 
Recently, Akbari and Kano~\cite{AK} considered the existence of
$\{k,\,r-k\}$-factors in $r$-regular graphs.

\begin{thm}[Akbari-Kano]\label{thm-AK}
Let $r$ and $k$ be integers such that $r$ is odd, $k$ is even
and $1\le k\le r$.
Then any $r$-regular graph has a $\{k,\,r-k\}$-factor.
\end{thm}

By Theorems~\ref{thm-Petersen}, \ref{thm-Tutte} and~\ref{thm-AK},
any $r$-regular graph has a
$\{k,\,r-k\}$-factor as if $k$ is even. 
For odd~$k$, 
Akbari and Kano~\cite{AK} posed the next problem for the case $r$ is even, 
and a conjecture for the case that $r$ is odd.

\begin{prb}[Akbari-Kano]\label{prb-AK}
Let $r$ and $k$ be integers such that $r$ is even, $k$ is odd and $1\le k\le r/2-1$. 
Is it true that every connected $r$-regular simple graph of
even order has a $\{k,\,r-k\}$-factor?
\end{prb}

Again, the ``even order'' condition is not a real restriction. 
On the other hand, any $r$-regular graph of even order has an $r/2$-factor.
This can be seen immediately from Theorem~\ref{thm-Gallai}
if one notices that any even-regular graph is $2$-edge connected.
Therefore, the condition $1\le k\le r/2-1$ is not a real restriction either.

The first aim of this paper is to give a negative answer to Problem~\ref{prb-AK}.
In Section~\ref{sec-prb-AK}, we construct an $r$-regular graph~$G^*$ without
$\{k,\,r-k\}$-factors for all $1\le k\le r/2-2$,
and deal with the case $k=r/2-1$ by using the following Lov\'asz's characterization~\cite{Lov72}
(see also~\cite[Theorem 6.1]{AK11B}) on parity-factors.
For any two subsets~$S$ and~$T$ of~$V(G)$, 
denote by $E_G(S,T)$ the set of edges 
with one end in $S$ and the other end in $T$. Denote
\[
e_G(S,T)=|E_G(S,T)|.
\]

\begin{thm}[Lov\'asz]\label{thm-Lovasz-parity}
Let $G$ be a graph, and $g,f\colon V(G)\to\mathbb{Z}$ be functions
such that $g(v)\le f(v)$ and
$g(v)\equiv f(v)\pmod 2$ for all vertices~$v$. Then
$G$ has a $(g,f)$-parity-factor if and only if
\begin{equation}\label{ineq-Lovasz-parity}
\eta(S, T)=\sum_{s\in S}f(s)+\sum_{t\in T}\bigl(d_G(t)-g(t)\bigr)-e_G(S,T)-q(S,T)\ge 0
\end{equation}
for all disjoint subsets $S$ and $T$ of $V(G)$, 
where $q(S, T)$ denotes the
number of components $C$ of the graph $G-S-T$ such that
\begin{equation}\label{eq-mod}
\sum_{c\,\in V(C)}f(c)+e_G(V(C),\,T)\equiv  1 \pmod 2.
\end{equation}
\end{thm}

In fact, Lov\'asz \cite{Lov72} presented a structural description for the degree
constrained subgraph problem for the case that no two consecutive integers are
missed in~$H(v)$ for every~$v$. He also showed that the problem without this
restriction is NP-complete. 
In particular, the next theorem, which is due to Lov\'asz~\cite{Lov70} 
(see also~\cite[Theorem 4.1]{AK11B}),
will be used in our deduction.

\begin{thm}[Lov\'asz]\label{thm-Lovasz-gf}
Let $G$ be a graph, and $g,\,f \colon  V (G)\to\mathbb{Z}$ be functions such that
$g(v)\le f(v)$ for all vertices $v$. Then $G$ has a $(g,f)$-factor if and
only if
\[
\gamma(S, T)=\sum_{s\in S}f(s)+\sum_{t\in T}\bigl(d_G(t)-g(t)\bigr)-e_G(S,T)-q^*(S, T) \ge0
\]
for all disjoint subsets $S$ and $T$ of $V(G)$, where $q^*(S, T)$ denotes
the number of components $C$ of the graph $G-S-T$ satisfying~(\ref{eq-mod}), and
$g(v)=f(v)$ for all $v \in V(C)$.
\end{thm}

By using Alon's combinatorial nullstellensatz~\cite{Alon99},
Shirazi and Verstra\"ete~\cite{Shi08} established the following brief result
for general $H$-factors,
which was originally posed by 
Addario-Berry et al.~\cite{ADMRT07} as a conjecture.

\begin{thm}[Shirazi-Verstra\"ete]\label{thm-SV}
Let $G$ be a graph with an associated set function~$H$. If 
\begin{equation}\label{cond-SV}
|H(v)|>\biggl\lceil {d_{G}(v)\over 2}\biggr\rceil\qquad\hbox{for all }v\in V(G),
\end{equation}
then $G$ has an $H$-factor.
\end{thm}

Frank et al.~\cite{FLS08} found an elementary proof 
for Theorem~\ref{thm-SV} by using the next result on directed graphs.
For any directed graph~$G$,
denote by~$d_G^-(v)$ the in-degree of~$v$.

\begin{thm}[Frank et al.]\label{thm-FLS}
Let $G$ be a graph with an associated set function $H$.
If $G$ has an orientation for which
\begin{equation}\label{ineq-FLS}
d_G^-(v) \ge |\{0,\,1,\,\ldots,\,d_{G}(v)\}\backslash H(v)|\qquad\hbox{for all }v\in V(G),
\end{equation}
then $G$ has an $H$-factor.
\end{thm}

It seems that the existence of $H$-factors in regular graphs 
has not been extensively investigated yet. 
Let $G$ be a graph, and $H$ a set function associated with~$G$.
Denote
\begin{align*}
mH&=\min_{v\in G}\min H(v),\\[5pt]
MH&=\max_{v\in G}\max H(v).
\end{align*}
Here is the second result of this paper.

\begin{thm}\label{thm-main}
Let $G$ be an $\{r,\,r+1\}$-graph with an associated set function~$H$.
If $mH\ge1$, $MH\le r$ and
\begin{equation}\label{cond-large}
|H(v)|\ge {MH-mH+3\over 2}\quad\hbox{for all }v\in V(G),
\end{equation}
then $G$ has an $H$-factor.
\end{thm}

The proof of Theorem~\ref{thm-main} will be given in Section~\ref{sec-H-regg}.
As will be seen, the condition~(\ref{cond-large}) is sharp.
For the case 
\[
H(v)=\{k,\,k+1\}\qquad\hbox{for all }v\in V(G), 
\]
where $1\le k\le r-1$,
Theorem~\ref{thm-main} reduces to Theorem~\ref{thm-Thomassen}. 
Moreover, as a result restricting to $\{r,\,r+1\}$-graphs,
Theorem~\ref{thm-main} is stronger than Theorem~\ref{thm-SV}
because the condition~(\ref{cond-SV}) implies~(\ref{cond-large})
for $\{r,\,r+1\}$-graphs.

\section{Answer to Akbari-Kano's problem}\label{sec-prb-AK}

This section is concerned with Problem~\ref{prb-AK}. 
Note that $1\le k\le r/2-1$.
The following theorem deal with the case $k\le r/2-2$.
For any integer~$n$, 
denote by $[n]_{\mathrm{odd}}$ the set of positive odd integers less than or equal to $n$.
For any vertex~$v$ in any graph~$G$, 
denote by~$N_G(v)$ the neighborhood of~$v$ in~$G$. 

\begin{thm}\label{thm-r/2-2}
For any even integer $r$,
there exists an $r$-regular graph $G^*$ of even order such that
$G^*$ has no $H^*$-factors where
\[
H^*=[r]_{\mathrm{odd}}\Big\backslash\Bigl\{{r\over 2}-1,\ {r\over 2},\ {r\over 2}+1\Bigr\}.
\]
In particular, $G^*$ has no $\{k,\,r-k\}$-factors for any odd integer $k$ such that $1\le k\le r/2-2$.
\end{thm}

\begin{proof}
Let~$J$ be the graph obtained by removing an edge from the complete graph~$K_{r+1}$. 
Let $J_1$, $J_2$, $\ldots$, $J_r$ be pairwise disjoint copies of~$J$. 
In each copy $J_i$,
let $a_i$ and $b_i$ be the ends of the edge that removed from $K_{r+1}$.
Let $G^*$ be the graph consisting of the copies $J_1$, $J_2$, $\ldots$, $J_r$,
together with two new vertices~$u$ and~$v$, 
such that 
\begin{align}
N_{G^*}(u)&=\bigl\{a_1,\ b_1,\ a_2,\ b_2,\ \ldots,\ a_{{r\over 2}-1},\ b_{{r\over 2}-1},\ a_{r-1},\ a_r\bigr\},
\label{Neighbor-u}\\[5pt]
N_{G^*}(v)&=\bigl\{a_{r\over 2},\ b_{r\over 2},\ a_{{r\over 2}+1},\ b_{{r\over 2}+1},\ \ldots,\ a_{r-2},\ b_{r-2},\ b_{r-1},\ b_r\bigr\}.\notag
\end{align}
Then~$G^*$ is an $r$-regular graph of the even order $r(r+1)+2$.

Now we shall show that $G^*$ has no $H^*$-factors.
Suppose to the contrary that $F$ is an $H^*$-factor of~$G^*$.
Let $1\le i\le r$.
Since $d_F(w)$ is odd for all $w\in J_i$, and the order~$|J_i|$ is odd, we find 
\begin{equation}\label{eq-F}
\sum_{w\in J_i}d_F(w)\equiv1\pmod2.
\end{equation}
Let $F_i$ be the subgraph of $F$ induced by the vertices in $J_i$.
By the Handshaking theorem, we have
\begin{equation}\label{eq-Fi}
\sum_{w\in J_i}d_{F_i}(w)\equiv0\pmod2.
\end{equation}
Taking the difference between~(\ref{eq-F}) and~(\ref{eq-Fi}), we obtain 
\[
e_F(J_i,\,\{u,v\})=\sum_{w\in J_i}\bigl(d_F(w)-d_{F_i}(w)\bigr)\equiv1\pmod2.
\]
Since $e_{G^*}(J_i,\,u)=2$ and $e_{G^*}(J_i,\,v)=0$ for $1\le i\le r/2-1$, we derive 
\[
e_F(J_i,\,u)=1\qquad\hbox{for }1\le i\le {r\over 2}-1.
\] 
By the definition~(\ref{Neighbor-u}) of~$N_{G^*}(u)$, we get
\[
d_{F}(u)\in\Bigl\{{r\over 2}-1,\ {r\over 2},\ {r\over 2}+1\Bigr\},
\]
contradicting the definition of~$H^*$. This completes the proof.
\end{proof}

The graph $G^*$ constructed above will be used to explain the sharpness
of the condition~(\ref{cond-large}) in the next section.
Now we cope with the case $k=r/2-1$.

\begin{thm}\label{thm-r/2-1}
Let $r$ be an even integer such that $r/2$ is even. 
Then any connected $r$-regular graph of even order
has an $\{r/2-1,\,r/2+1\}$-factor.
\end{thm}

\begin{proof}
We shall apply Theorem~\ref{thm-Lovasz-parity} by setting 
$g(v)=r/2-1$ and $f(v)=r/2+1$
for all vertices~$v$.
Let $G$ be a connected $r$-regular graph of even order.
Let~$S$ and~$T$ be disjoint subsets of $V(G)$.
First, we claim that
\begin{equation}\label{ineq-2q}
e_G(S\cup T,\, V(G)\backslash S\backslash T)\ge 2\,q(S,T).
\end{equation}
In fact, if $S\cup T\in\{\emptyset,G\}$, then 
$q(S,T)=0$,
and (\ref{ineq-2q}) follows immediately. 
Otherwise, let $C$ be a component of $G-S-T$. 
Then both $S\cup T$ and $C$ are nonempty.
Note that any even-regular graph is 2-edge-connected.
So $G$ is 2-edge-connected. In particular, we have
\[
e_G(S\cup T,\,C)\ge2.
\]
Summing the above inequality over all components $C$, 
we get the desired inequality~(\ref{ineq-2q}). 
Hence,
\begin{align*}
\eta(S, T)
&=\Bigl({r\over 2}+1\Bigr)\bigl(|S|+|T|\bigr)-e_G(S,T)
-q(S,T)\\[5pt]
&\ge{r\over 2}\bigl(|S|+|T|\bigr)-e_G(S,T)
-\frac{1}{2}e_G(S\cup T,\,V(G)\backslash S\backslash T)\\[5pt]
&=e_G(S,S)+e_G(T,T)\ge 0.
\end{align*}
By Theorem~\ref{thm-Lovasz-parity}, $G$ has an $\{r/2-1,\,r/2+1\}$-factor. 
\end{proof}

Combining Theorems~\ref{thm-r/2-2} and~\ref{thm-r/2-1},
we obtain a negative answer to Problem~\ref{prb-AK}.

\section{The existence of $H$-factors in regular graphs}\label{sec-H-regg}

This section is devoted to establish Theorem~\ref{thm-main}.
A subset $U$ of $V(G)$ is called {\it independent} if any two vertices in $U$
are not adjacent in~$G$. We need the following lemma to prove Theorem~\ref{thm-main}.

\begin{lem}\label{lem-U}
Let $r$ and $k$ be positive integers such that $1\le k\le r-1$. 
Let $G$ be an $\{r,\,r+1\}$-graph and
\[
U=\{v\in V(G)\ |\ d_G(v)=r+1\}.
\]
If $U$ is independent,
then $G$ has a $\{k,\,k+1\}$-factor $F$
such that 
\[
d_F(u)=k+1\qquad\hbox{as if}\quad u\in U.
\]
\end{lem}

\begin{proof} 
Let $f(v)=k+1$ 
for all vertices~$v$, and
\[
g(v)=\begin{cases}
k+1,&\hbox{if }v\in U,\\[5pt]
k,&\hbox{otherwise}.
\end{cases}
\]
It suffices to show that $G$ has a $(g,f)$-factor.
Suppose to the contrary that $G$ has no $(g,f)$-factors.
By Theorem~\ref{thm-Lovasz-gf}, we have
\[
\gamma(S,T)<0\qquad\hbox{for some }S,T\subseteq V(G).
\]
Let~$S$ and~$T$ be disjoint subsets of~$V(G)$ such that $\gamma(S, T)<0$ 
and the set $S\cup T$ is maximal.
We claim that $q^*(S,T)=0$.

Suppose to the contrary that $q^*(S,T)\ge1$. 
Let~$C$ be a component of $G-S-T$ counted by~$q^*(S,T)$.
It follows that
\begin{equation}\label{eq-component}
e_G(C,\,G-S-T)=0.
\end{equation}
By the definition of $q^*(S,T)$, we have
\begin{equation}\label{eq-g=f}
g(v)=f(v)=k+1\qquad\hbox{for all }v\in V(C).
\end{equation}
So $V(C)\subseteq U$. But $U$ is independent, 
we deduce that $C$ is a single vertex, say, $V(C)=\{a\}$.
Let $S'=S\cup\{a\}$ and $T'=T\cup\{a\}$. 
Then~(\ref{eq-component}) implies 
\begin{align}
q^*(S',T)=q^*(S,T)-1,\label{eq-q*S'}\\[5pt]
q^*(S,T')=q^*(S,T)-1.\label{eq-q*T'}
\end{align}
Note that the condition~(\ref{eq-mod}) implies $e_G(a,T)\ne k+1$.
If $e_G(a,T)\le k$, then~(\ref{eq-component}) and~(\ref{eq-g=f}) yield
\[
d_G(a)-e_G(a,S)=e_G(a,T)\le g(a)-1.
\]
Together with~(\ref{eq-q*T'}), we have
\[
\gamma(S,T')-\gamma(S,T)
=d_G(a)-g(a)-e_G(S,a)-q^*(S, T')+q^*(S,T)\le 0.
\]
So $\gamma(S,T')<0$, contradicting the maximality of $S\cup T$.
Otherwise $e_G(a,T)\ge k+2$. By~(\ref{eq-q*S'}), we deduce
\[
\gamma(S',T)-\gamma(S,T)
=f(a)-e_G(a,T)-q^*(S', T)+q^*(S, T) 
\le 0.
\]
So $\gamma(S',T)<0$, contradicting, again, the maximality of $S\cup T$.
Thus the claim is true.

Now we can deduce 
\begin{align*}
\gamma(S, T)
&=\sum_{s\in S}d_{G}(s)\frac{f(s)}{d_{G}(s)}+
\sum_{t\in T}d_{G}(t)\left(1-\frac{g(t)}{d_{G}(t)}\right)-e_G(S,T)\\
&\ge\sum_{\substack{ s\in S, \ t\in T\\ st\in E(G)}}
\left(\frac{f(s)}{d_{G}(s)}+\Bigl(1-\frac{g(t)}{d_{G}(t)}\Bigr)\right)
-e_G(S,T)\\
&=\sum_{\substack{ s\in S, \ t\in T\\ st\in E(G)}}
\left(\frac{k+1}{d_{G}(s)}-\frac{g(t)}{d_{G}(t)}\right)\\
&\ge \sum_{\substack{ x\in S, \ y\in T\\ xy\in E(G)}}
\left(\frac{k+1}{r+1}-\max\left(\frac{k}{r},\frac{k+1}{r+1}\right)\right)=0,
\end{align*}
contradicting the hypothesis $\gamma(S,T)<0$.
This completes the proof. 
\end{proof}

We remark that Lemma~\ref{lem-U} is a generalization of Theorem~\ref{thm-Tutte}.
Now we are in a position to prove Theorem~\ref{thm-main}.

\begin{proof}
Write $m=mH$ and $M=MH$ for short.
By Theorem~\ref{thm-Thomassen},
we can suppose that $F$ is an $\{M,\,M+1\}$-factor of $G$ 
with the minimum number of edges.
It follows that any two vertices of degree $M+1$ in $F$, if they exist, are not adjacent.
By Lemma~\ref{lem-U}, 
$F$ has an $\{m-1,\,m\}$-factor, say,~$F'$,
such that 
\begin{equation}\label{eq-F'}
d_{F'}(v)=m\quad\hbox{as if}\quad d_F(v)=M+1.
\end{equation}
Let $F''$ be the complemented graph of $F'$ in $F$. 
In view of~(\ref{eq-F'}),
we have
\begin{equation}\label{deg-F''}
d_{F''}(v)\in\{M-m,\,M-m+1\}\qquad\hbox{for all }v.
\end{equation}
We observe that $F''$ has an orientation such that 
\begin{equation}\label{ineq-indegree}
d_{F''}^-(v)\ge\Bigl\lfloor{d_{F''}(v)\over 2}\Bigr\rfloor
\qquad\hbox{for all }v.
\end{equation}
This can be seen by orienting an eulerian tour of the graph that obtained from~$F''$ 
by adding a new vertex and joining it to all vertices of odd degree in~$F''$.
Let
\[
H'(v)=\{h-d_{F'}(v)\ |\ h\in H(v)\}\qquad\hbox{for all }v.
\]
Then the condition~(\ref{cond-large}) reads 
\begin{equation}\label{ineq-H'}
|H'(v)|=|H(v)|\ge{M-m+3\over 2}.
\end{equation}
By~(\ref{deg-F''}), (\ref{ineq-indegree}) and~(\ref{ineq-H'}),
it is easy to verify that
\[
|\{0,\,1,\,\ldots,\,d_{F''}(v)\}\backslash H'(v)|\le d_{F''}^-(v)\qquad\hbox{for all }v.
\]
By Theorem~\ref{thm-FLS}, the graph~$F''$ has an $H'$-factor, say, $G'$.
Hence, 
the graph induced by the edge set $E(F')\cup E(G')$ is an $H$-factor of $G$. 
This completes the proof. 
\end{proof}

In fact, the condition~(\ref{cond-large}) is sharp.
For instance,
when $r$ is even,
let $G^*$ be the graph constructed in the proof of Theorem~\ref{thm-r/2-2}.
Consider a set $H$ of the form
\[
H=\{m,\,m+2,\,m+4,\,\ldots,\,M\},
\]
where both $m$ and $M$ are odd, and $M\le r/2-2$. 
On one hand, $G^*$ has no $H$-factors by Theorem \ref{thm-r/2-2}.
On the other hand, it is straightforward to compute
\[
|H|={M-m+2\over 2}.
\]
Comparing it with the condition~(\ref{cond-large}),
we deduce the latter one is sharp.
For other possibilities of the associated set~$H$, 
for example, $mH+MH$ is odd, we mention that
it is also not hard to find $r$-regular graphs without $H$-factors such that 
\[
|H(v)|=\left\lfloor{MH-mH+2\over 2}\right\rfloor\qquad\hbox{for all }v\in V(G).
\]

\noindent{\bf Acknowledgments.} 
Lu was supported by the National Natural Science Foundation of China (Grant No. 11101329) and the Fundamental Research Funds for the Central Universities. Wang was supported by the National Natural Science Foundation of China (Grant No. 11101010). 
We are grateful to Mikio Kano for sharing the $\{k,\,r-k\}$-factor problem with us.

\end{document}